%% file: etrscurrent.tex
\documentclass[12pt]{article}
\usepackage[left=1in,top=1in,right=1in,bottom=1in,letterpaper]{geometry}
\usepackage{setspace}
\usepackage{amssymb}
\usepackage{mathrsfs}
\usepackage{algorithmic}
\usepackage{multirow}
\usepackage{rotating}
\usepackage{cite}
\usepackage{lscape}
\usepackage{caption}
\usepackage{amsmath}
\usepackage{makeidx}
\usepackage{cite}
\usepackage{mdframed}
\usepackage{amsmath,amstext,amsthm} % Lots of math symbols and environments
\usepackage{seceqn}
\makeindex

\usepackage{graphicx}
\usepackage{color}
\newtheorem{definition}{Definition}[section]

\newtheorem{assumption}{Assumption}[section]

\newtheorem{theorem}{Theorem}[section]

\newtheorem{corollary}{Corollary}[section]

\newtheorem{remark}{Remark}[section]

\newtheorem{algorithm}{Algorithm}[section]

\newtheorem{lemma}{Lemma}[section]

\newcommand{\R}{\mathds R}

\newcommand{\A}{\mathcal{A}}
\newcommand{\OO}{\mathcal{O}}
\newcommand{\On}{\OO^n}
\font\transp=msbm10 scaled\magstep 1 
scaled\magstep 0
\def\transpR{\transp\char'122}
\def\nameuse#1{\csname#1\endcsname}
\def\R{\nameuse{@ifnextchar} [{\Rmv}{\mbox{\transpR}}}
\def\Rmv[#1]{\nameuse{@ifnextchar} [{\Rmat{#1}}{\Rvec{#1}}}
\def\Rvec#1{\mbox{\transpR}^{#1}}
\def\Rmat#1[#2]{\mbox{\transpR}^{{#1} \times {#2}}}
\def\@begintheorem#1#2{\par\bgroup{\bf #1\ #2. }\it\ignorespaces}
\def\@opargbegintheorem#1#2#3{\par\bgroup{\bf #1\ #2\ (#3). }\it\ignorespaces}
\def\@endtheorem{\egroup}

\def\R{\mathbb{R}}
\def\Rn{\R^n}

\def\Rpp{\R_{++}}
\def\S{\mathbb{S}}
\def\Sn{\S^n}
\newcommand{\textdef}[1]{\textit{#1}\index{#1}}
\newcommand{\LNGM}{\textbf{LNGM}\, }
\newcommand{\LNGMp}{\textbf{LNGM}}
\newcommand{\SQP}{\textbf{SQP}\, }
\newcommand{\SQPp}{\textbf{SQP}}

\newcommand{\SCQp}{\textbf{SCQ}}
\newcommand{\LICQ}{\textbf{LICQ}\, }
\newcommand{\LICQp}{\textbf{LICQ}}
\newcommand{\TRSproj}{\textbf{TRS$_{proj}$}\, }
\newcommand{\TRSprojp}{\textbf{TRS$_{proj}$}}
\newcommand{\Socsdp}{\textbf{SOCP/SDP}\, }
\newcommand{\Socsdpp}{\textbf{SOCP/SDP}}
\newcommand{\TRS}{\textbf{TRS}\, }
\newcommand{\TRSp}{\textbf{TRS}}
\newcommand{\TRSmu}{\textbf{TRS$_{\mu}$}\, }
\newcommand{\TRSmup}{\textbf{TRS$_{\mu}$}}
\newcommand{\nTRS}{\textbf{nTRS}\, }
\newcommand{\nTRSp}{\textbf{nTRS}}
\newcommand{\sTRS}{\textbf{sTRS}\, }
\newcommand{\sTRSp}{\textbf{sTRS}}
\newcommand{\eTRS}{\textbf{eTRS}\, }
\newcommand{\eTRSp}{\textbf{eTRS}}
\newcommand{\KKTo}{\textbf{KKT1}\, }

\newcommand{\KKTt}{\textbf{KKT2}\, }

\DeclareMathOperator{\Diag}{Diag}
\DeclareMathOperator{\Range}{Range}
\DeclareMathOperator{\Rank}{rank}
\DeclareMathOperator{\Null}{Null}
\DeclareMathOperator{\spanl}{span}
\begin{document}
\title{
Local Nonglobal Minima for Solving\\
Large Scale Extended Trust Region Subproblems
}
%\titlerunning{Trust Region Subproblem and Linear Constraint}
\author{Maziar Salahi
\thanks{Faculty of Mathematical Sciences, University of Guilan, Rasht,
Iran.  Thanks to the University of Guilan for
supporting the sabbatical leave 2015-16,  hosted by Prof. H. Wolkowicz at
the Department of Combinatorics and Optimization, University of Waterloo.
}
\and Akram Taati 
\thanks{Faculty of Mathematical Sciences, University of Guilan, Rasht,
Iran.}
\and Henry Wolkowicz
\thanks{Faculty of Mathematics, University of Waterloo, Canada.
Supported in part by NSERC and AFOSR grants.}
} 
%\institute{M. Salahi \and A. Taati \at
%%              Faculty of Mathematical Sciences, University of Guilan, Rasht, Iran \\
%%              Tel.: +98-133333901\\
%%              Fax: +98-1333333509\\
%%              \email{salahim@guilan.ac.ir, akramtaati@yahoo.com}  
%%	      \\H. Wolkowicz at \at Faculty of Mathematics, University
%%	      of Waterloo,
%%	      \email{hwolkowicz@uwaterloo.ca}
%%      }         %  \\
%%%             \emph{Present address:} of F. Author  %  if needed
%%

\maketitle
\abstract{
We study large scale extended trust region subproblems (\eTRSp) i.e.,~the 
minimization of a general quadratic function
subject to a norm constraint, known as the trust region subproblem
(\TRSp) but with an additional linear inequality
constraint. It is well known that strong duality holds for the \TRS and
that there are efficient algorithms for solving large scale \TRS
problems. It is also known that there can exist at most one local
non-global minimizer (\LNGMp) for \TRSp. We combine this with known
characterizations for strong duality for \eTRS and, in particular,
connect this with the so-called \emph{hard case} for \TRSp.

We begin with a recent characterization of the minimum for the \TRS via a
generalized eigenvalue problem and extend this result to the \LNGMp. We
then use this to derive an efficient algorithm that finds the global
minimum for \eTRS by solving at most three generalized eigenvalue
problems.
}

~
{\bf Classification code:} 90C26, 90C30, 90C46, 65F15

%\hspace{-1cm}\rule{\textwidth}{0.2mm}
~~\newline
\tableofcontents
\listoftables
%\listofalgorithms
%\listoffigures
%\addtocontents{toc}{\noindent subsection \par}

%%%%%%%%%%%%%%%%%%%%%%%%%%%%%%%%%%%%%%%%%%%%%%%%%%

%%%%%%%%%%%%%%%%%%%%%%%%%%%%%%%%%%%%%%%%%%%%%%%%%%
\section{Introduction}
We study large scale instances of 
the \textdef{extended trust region subproblem, \eTRS}
\index{\eTRSp, extended trust region subproblem}
\begin{equation}
	\label{eTRS}
	\tag{eTRS}
	\begin{array}{rcl}
		p^{*}:=&\min  &\textdef{$q(x):=x^TAx +2a^Tx$} \\
	& \text{s.t.} &\textdef{$g(x):=||x||^2-\delta$}\leq 0, \\
		       &&\textdef{$\ell(x):=b^Tx-\beta$}\leq 0, 
\end{array}
\end{equation}
where $A\in \Sn$ is a real $n\times n$ symmetric matrix,  
$a,b\in\Rn\backslash \{0\}$ and $\beta\in\R, {\delta\in\Rpp}$.
\index{$\Sn$, symmetric real $n\times n$ matrices}
\index{symmetric real $n\times n$ matrices, $\Sn$}
Here a linear inequality constraint is added onto the standard
\textdef{trust region subproblem, \TRSp}.
\index{\TRSp, trust region subproblem}
The \TRS is an important subproblem in trust region methods for both
constrained and unconstrained problems, e.g.~\cite{ConGouToi:00}.
The \eTRS problem extends the \TRS and is a step toward solving \TRS
with a general number of inequality constraints. Such problems would be
useful for example in the subproblem of finding search directions for sequential
quadratic programming (\SQPp) methods for general nonlinear programming,
e.g.,~\cite{BoTo:95}.

It is known that, surprisingly,
strong duality holds for \TRS and the global minimizer
can be found efficiently and accurately, even though the objective
function is not necessarily convex. The early algorithms for 
moderate sized problems are based on exploiting the positive
semidefinite second order optimality conditions using a Cholesky factorization 
of the Lagrangian, see e.g.,~\cite{MoSo:83,Gay:81}. These methods were
extended to the large scale case using a parametrized eigenvalue
problem, e.g.~\cite{ReWo:94,LRSV:11,Hager:04,GoRoTho10}.
A related problem is finding the local non-global minimizer (\LNGMp) of \TRS
if it exists, see~\cite{Mar:94}. See~\cite{ConGouToi:00} for
more extensive details, applications, and background for \TRSp.

However, strong duality can fail for the \eTRSp. 
This is characterized in \cite{MR2486048} for the more general two
quadratic constraint problem. We show that this is exactly connected to
the so-called \emph{hard case} for \TRSp.
We use this fact to find an efficient approach for finding the
global minimizer for \eTRSp.
Recently, a generalized eigenvalue characterization for the \TRS 
optimum is derived in Adachi et al
\cite{AdachiIwataNakatsukasaTakeda:15} based on solving a
\emph{single} generalized eigenvalue problem.
This algorithm is shown to be extremely efficient for solving the \TRSp.
In this paper we extend this result for the \LNGM
optimum using the second largest real generalized eigenvalue of a matrix
pencil. This provides an efficient procedure for finding the \LNGMp.
From combining the solutions for \TRS and \LNGM we derive an efficient
algorithm for \eTRSp.

We include a discussion relating strong duality and stability for \eTRSp.
Extensive numerical tests show that our new algorithm is accurate and
can solve large scale problems efficiently.

Related previous work on strong duality and an eigenvalue approach 
on \eTRS appeared
in e.g.,~\cite{SalahiFallahi2015,SalahiTaati2015,MR3258522,HsiaSheu:13}.

\subsection{Notation and Preliminaries} 
We let
\[
\textdef{$\lambda_{\min}(A)=\lambda_1\leq 
\lambda_2 \leq \ldots \leq \lambda_n$}
\]
denote the eigenvalues of $A$ in nondecreasing order, and
$A=Q\Lambda Q^T$ be the \textdef{orthogonal spectral decomposition of
$A$} with the diagonal matrix of eigenvalues
$\Lambda=\Diag(\lambda_1, ..., \lambda_n)$. 
We denote the \textdef{orthogonal matrices, $\On$}. We let
$q_i$ denote the orthonormal columns of the eigenvector matrix $Q\in \On$. 
\index{$\On$, orthogonal matrices}

For $X\in \Sn$ the space of $n\times n$ real symmetric matrices,
we let \textdef{$X\succeq 0, \succ 0$} denote positive semidefiniteness
and definiteness, respectively.
In addition, we define the \textdef{vector of ones, $e$} of appropriate
size and $\Diag(v)$ be the diagonal matrix formed from the vector $v$.
\index{$e$, vector of ones}

It is now well known that, surprisingly,
the possibly nonconvex \TRS problem has the following characterization
of optimality with a positive semidefinite Lagrangian Hessian.
\begin{theorem}[\textdef{Characterization of Global Minimum of 
	\TRSp},~\cite{MoSo:83,Gay:81}]
	\label{thm:globalTRS}
	Define the 
	\begin{equation}
		\label{eq:lagrtrs}
		\textdef{Lagrangian of \TRSp, $L(x,\lambda)$}
	:=q(x)+\lambda (\|x\|^2-\delta).
		\index{$L(x,\lambda)$, Lagrangian of \TRSp}
\end{equation}
	\index{$L(x,\lambda)$, Lagrangian of \TRSp}
	The vector $x^* \in \Rn$ is a global optimum of \TRS if, and
	only if, there exists $\lambda \in \R$ such that
	\[
	\begin{array}{rcrcl}
\frac 12\nabla L(x^*,\lambda^*)&=&(A+\lambda^* I)x^*+ a & = & 0, \quad \lambda^* \geq 0 \\
\frac 12	\nabla^2 L(x^*,\lambda^*) &=&A+\lambda^* I & \succeq &  0\\
		&&\|x^*\|^2-\delta  & \leq &  0   \\
			&&\lambda(\|x^*\|^2-\delta)  & = & 0   \\
	\end{array}
	\]
	\qed
\end{theorem}
Now if $x^*$ is a global minimizer of \TRS and $\nabla^2
L(x^*,\lambda^*)$ is singular, then 
\[
	\lambda^*=-\lambda_1 \text{   and   } 0\neq a\in \Range(A+\lambda^*I)= 
	(\Null(A+\lambda^* I))^\perp 
\]
holds and leads to the following definition.
\begin{definition}[Hard Case]
	The \textdef{hard case} holds for \TRS if $a$ is orthogonal to
	the eigenspace corresponding to $\lambda_1$, 
	$\Null(A+\lambda^* I)$.
\end{definition}

In addition, the \textdef{Slater constraint qualification, \SCQp}, 
or strict feasibility, can be assumed without loss of generality for
feasible instances of \eTRSp. 
\begin{lemma}
	\label{lem:slater}
The \eTRS is feasible, respectively strictly feasible, if, and only if
\begin{equation}
	\label{eq:slatercqiff}
	-\sqrt \delta \|b\| \leq \beta,  \text{    respectively  }
	-\sqrt \delta \|b\| < \beta. 
\end{equation}
Moreover, if equality holds on the left 
in \eqref{eq:slatercqiff}, then \eTRS has
the unique feasible (and so optimal) point $x^*=-\frac {-\sqrt
\delta}{\|b\|} b$.
\end{lemma}
\begin{proof}
Consider the problem $\min_x \{x^Tb : \|x\|^2\leq \delta\}$.
We can differentiate the Lagrangian to get  
\[
	0\neq x=\frac {-1}{2\lambda}b, \lambda >0.
\]
Since $x^Tb= \frac {-1}{2\lambda}b^Tb<0$, the minimum value is obtained
with $0<\lambda$ small. We now have
\[
	x^Tx= \frac {1}{(2\lambda)^2}\|b\|^2 \leq \delta
	\implies 2\lambda = \frac {\|b\|}{\sqrt \delta}.
\]
The result now follows by noting that the linear inequality constraint is
\[
	x^Tb=  -\frac {1}{2\lambda}\|b\|^2 \leq \beta
\]
and then substituting for the value found for $2\lambda$.
\end{proof}

We note that if the global solution of \TRS is feasible for
\eTRS then it is clearly optimal. And from the above, we know that it
can be found efficiently using a generalized eigenvalue problem.
Therefore from this and Lemma \ref{lem:slater} we can
make the following assumption for the theoretical part of the paper. (We
do not make this assumption for the algorithmic part.)
\begin{assumption}
	\label{assum:slater}
We assume in this paper that \eTRS is strictly feasible and that the
global solution of \TRS is infeasible for \eTRSp.
\end{assumption}

\subsection{Outline}
We continue in Section \ref{sect:lngm} with the details on the \LNGMp.
This includes known results from \cite{Mar:94} and
one of the main results of this paper in Theorem
\ref{mainlngmpeig}, the necessary conditions for a \LNGM using the
second largest real generalized eigenvalue of a matrix pencil. In Section
\ref{sect:charact} we discuss necessary and sufficient conditions for
strong duality to hold for \eTRSp.  A discussion on the stability of
\eTRS and resulting stability of our approach is given in
Section \ref{sect:stablty}.

The various optimality conditions for \eTRS are applied in Section
\ref{sect:globetrs}.  Included in this section are
outlines of the algorithms for an
efficient numerical procedure to find the global optimum of \eTRSp.
Our numerical results appear in Section \ref{sect:nums}. We provide
concluding remarks in Section \ref{sect:concl}.

\section{On a Local Non-global Minimizer (\LNGMp) of \TRS}
\label{sect:lngm}

\subsection{Background on \LNGMp}
\label{sect:backgr}
Let $x^{*}$ be a global optimal solution of \eTRSp. Then 
the linear constraint is either inactive $b^Tx^{*}<\beta$ or active
$b^Tx^{*}=\beta$.
In the former case, we have $x^{*}$ must be a local (not global
by Assumption \ref{assum:slater}) 
minimizer of \TRSp, i.e.,~we can have $x^*$ being a
\textdef{local non-global minimizer, \LNGMp}, of \TRSp.
\index{\LNGMp, local non-global minimizer}
We now provide some background on the \LNGMp.
\begin{lemma}\label{l11}
	If $A\succeq 0$, then no \LNGM exists.
\end{lemma}
\begin{proof}
	This is immediate since $A\succeq 0$ implies that \TRS is a
	convex problem, i.e.,~a problem where local minima are global
	minima. (It also follows from Theorem \ref{thm:lngmmain} below, since 
	$0\leq \lambda^* < -\lambda_1$.)
\end{proof}

Therefore, in this section we assume that $\lambda_1<0$. 
We continue and present some known results related to \LNGMp.
Then following the results in
\cite{AdachiIwataNakatsukasaTakeda:15}, we
show that the \LNGM can be computed via a generalized eigenvalue problem.

\begin{theorem}
	[\textdef{Necessary Conditions for \LNGMp}, {\cite{Mar:94}}]
	\label{thm:lngmmain}
	Let $x^{*}$ be a \LNGMp. Choose $V \in \R^{n\times (n-1)}$ such
	that
	$\left[ \begin{array}{c|c}
			\frac 1{\|x^*\|} x^* ~&~ V
	\end{array} \right] \in \On$.
	Then there exists a unique $\lambda^{*} \in 
	(\max\{0,-\lambda_2\}, -\lambda_1)$\footnote{\label{foot:pos}
		We have added the fact that $\lambda^*>0$
		whereas only nonnegativity is given in
		\cite[Theorem 3.1]{Mar:94}. Strict complementarity is
	proved in \cite[Prop. 3.4]{LuPaRo:96}. In fact, it is easy to
	see by the second order conditions that
	strict complementarity holds as well for the global
minimum for \TRS in the $\lambda_1<0$ case.}
	such that
\begin{align}\label{eq}
	V^T(A+\lambda^{*} I)V \succeq 0, \qquad
(A+\lambda^{*} I) x^{*}=-a, \qquad  ||x^{*}||^2=\delta.
\end{align}
\qed
\end{theorem}
\begin{corollary}\label{cor:hardcase}
	If the so-called \textdef{hard case} holds for \TRSp, i.e.,~$a$
	is orthogonal to the eigenspace corresponding to $\lambda_1$,
	then no \LNGM exists.\footnote{The hard case arises in
	algorithms for \TRSp. The singularity that can arise
	requires special treatment, see e.g.,~\cite{MoSo:83}. In fact,
	it can be handled by a shift and deflation step,
see~\cite{FortinWolk:03}.}
\end{corollary}
\begin{proof}
The proof is given in {\cite[Lemma 3.2]{Mar:94}}. 
We include a separate proof to emphasize that a stronger result holds
as is given in Corollary \ref{cor:hardcasesimple} below.

After a rotation if needed, we can assume for simplicity that 
$A=\textdef{$\Diag$}(\lambda)$ is a diagonal matrix. To obtain a
contradiction, we assume that $a^Tq_1=0$. From this assumption
we have that the first element $a_1=0$. From  \eqref{eq} this implies
that the first element $x_1^*=0$ which yields that the first eigenvector
given by the first unit vector $e_1$ satisfies $e_1=Vu$, for some $u$.
We have $u^TV^T(A+\mu I)Vu = \lambda_1+\lambda^*<0$. This contradicts
the second order semidefiniteness condition in \eqref{eq}.
\end{proof}
\begin{corollary}\label{cor:hardcasesimple}
If the weak form of the \textdef{hard case} holds for \TRSp, i.e.,~$a$
is orthogonal to \underline{some} eigenvector corresponding to $\lambda_1$,
then no \LNGM exists.
\end{corollary}
\begin{proof}
	The proof of Corollary \ref{cor:hardcase} just needed one eigenvector orthogonal to $a$.
\end{proof}

Now let
\index{$\phi(\lambda):=\left({\rm norm}((A+\lambda I)^{-1}a)\right)^2$}
 \begin{align*}
\phi(\lambda):=\| (A+\lambda I)^{-1}a)\|^2.
\end{align*}
For 
\[
	\lambda \in (\max\{0,-\lambda_2\}, -\lambda_1), 
\]
Theorem \ref{thm:lngmmain} shows that the equation $\phi(\lambda)=\delta$
is a necessary condition for a \LNGMp.
Furthermore, using  the eigenvalue decomposition of $A$ we have
\begin{align}\label{eh}
\phi(\lambda)&=\sum_{i=1}^n \frac{(q_i^Ta)^2}{(\lambda_i+\lambda)^2},\notag \\
\phi'(\lambda)&=-2\sum_{i=1}^n\frac{(q_i^Ta)^2}{(\lambda_i+\lambda)^3},\\
\phi''(\lambda)&=6\sum_{i=1}^n\frac{(q_i^Ta)^2}{(\lambda_i+\lambda)^4}.\notag
\end{align}
The equations (\ref{eh}) imply that the function $\phi(\lambda)$ is strictly 
convex on $\lambda \in (\max\{0,-\lambda_2\}, -\lambda_1)$ and so it has at 
most two roots in the interval $(\max\{0,-\lambda_2\}, -\lambda_1)$.
 The following theorem states that  only the largest root can correspond
 to a \LNGMp.
 \begin{theorem}(\cite[Theorem 3.1]{Mar:94})
	 \label{lngmone}
	 \begin{enumerate}
		 \item
	If $x^*$ is a \LNGMp, then \eqref{eq} holds with a unique
	$\lambda^{*}\in (\max\{0,-\lambda_2\}, -\lambda_1)$ 
	and with $\phi'(\lambda^{*})\geq 0$. 
\item
	There exists at most one \LNGMp.
	 \end{enumerate}
	 \qed
\end{theorem}

\subsection{Characterization using a Generalized Eigenvalue Problem}
We now consider the problem of efficiently computing the \LNGMp.
Due to the results in Section \ref{sect:backgr} we can make the
following two assumptions. 
\begin{assumption}
	\label{assum:indefhard}
\begin{enumerate}
	\item  The smallest two eigenvalues of $A$ satisfy
		\[
			\lambda_1<\min\{0,\lambda_2\}.
		\]
	\item
		The hard case does \emph{not} hold, i.e.,~$a$ is
		\emph{not} orthogonal to the eigenspace corresponding to
		$\lambda_1$ which here is $\spanl(q_1)$ the span of the
		eigenvector of $\lambda_1$, $a^Tq_1\neq 0$.
\end{enumerate}
\end{assumption}

To the best of our knowledge, the only algorithm for computing the \LNGM
is the one by Martinez \cite{Mar:94} which tries to find the largest
root of the equation $\phi(\lambda)=\delta$ for $\lambda \in
(\max\{0,-\lambda_2\},-\lambda_1)$ via an iterative algorithm. Each step of his
approach requires solving an indefinite system of linear equations which
can be expensive for large scale instances. In what follows, we follow
on the ideas of \cite{AdachiIwataNakatsukasaTakeda:15} and present
a new algorithm that shows that
the \LNGM can be computed efficiently by a generalized eigenvalue problem. 
Our result is then used to solve large instances of \eTRSp.

Recently, Adachi et al.~\cite{AdachiIwataNakatsukasaTakeda:15}
designed an efficient algorithm for \TRS
which solves just one generalized eigenvalue problem.
They consider the following $2n\times 2n$ regular symmetric
matrix pencil which has $2n$ finite eigenvalues.\footnote{The
objective function in~\cite{AdachiIwataNakatsukasaTakeda:15} is $1/2$
our objective function and $I$ in the pencil is a general $B\succ 0$.}
$$ 
M(\lambda)=\begin{bmatrix}-I & A+\lambda I\\A+\lambda I& 
                         -\frac1 {\delta}{aa^T}\end{bmatrix}.
$$
We can rephrase Theorem \ref{thm:globalTRS} as
	$x_g^{*}$ is a global optimal solution of \TRS if, and only
if, the following system is consistent.
\begin{subequations}
	\label{eq:optTRS}
\begin{align}
&(A+\lambda_g^{*} I) x_g^{*}=-a, \label{4}\\
& A+\lambda_g^{*}I\succeq 0, \quad  \lambda_g^{*}\geq 0 \text{  and
unique}, \label{7} \\
& ||x_g^{*}||^2\leq \delta, \label{5}\\
& \lambda_g^{*}(||x_g^{*}||^2-\delta)=0.\label{6}
\end{align}
\end{subequations}
\begin{lemma}[\textdef{Generalized Eigenvalue of Pencil}, \cite{AdachiIwataNakatsukasaTakeda:15}]
	\label{lem:geneigtrs}
	For every Lagrange multiplier
	$\lambda_g^*\neq 0$, satisfying the stationarity condition \eqref{4}
	with equality in the quadratic constraint \eqref{5},
we have $\det M(\lambda_g^*)=0$, i.e.,~$\lambda^*_g$
is a \textdef{generalized eigenvalue of the pencil $M(\lambda)$}.
	\index{KKT, Karush-Kuhn-Tucker}
\end{lemma}
\begin{proof}
The Lemma is proved in \cite{AdachiIwataNakatsukasaTakeda:15}. We 
include a shorter proof.

For simplicity we denote $D=A+\lambda I$ and
let $\lambda=\lambda^*_g$ be a Lagrange multiplier  satisfying (\ref{4}).
We can rewrite (with $x=x_g^*$)
\begin{equation}
	\label{eq:detMfact}
\begin{bmatrix}I & 0 \\ 0  & D \end{bmatrix}
\begin{bmatrix}-I & I\\I& -\frac1 {\delta}{xx^T}\end{bmatrix}
\begin{bmatrix}I & 0 \\ 0  & D \end{bmatrix} =
\begin{bmatrix}I & 0 \\ 0  & D \end{bmatrix}
\begin{bmatrix}-I & D\\I& -\frac1 {\delta}{xx^T}D\end{bmatrix}
=M(\lambda).
\end{equation}
The result follows by observing that the vector 
$0\neq \begin{pmatrix}x \cr  x  \end{pmatrix}\in \Null 
\left(\begin{bmatrix}-I & I\\I& -\frac1
{\delta}{xx^T}\end{bmatrix}\right)$.
\end{proof}
\begin{corollary}
	The set of real generalized eigenvalues of $M(\lambda)$ is
	nonempty.
	Moreover, if $\det M(\lambda)=0, \lambda \in \R$, then either
	$-\lambda$ is an eigenvalue of $A$ or
	\[
	\det\left(\begin{bmatrix}-I & I\\I& -\frac1 {\delta}{xx^T}\end{bmatrix}
	\right)=0, \quad x=-(A+\lambda I)^{-1} a.
	\]
\end{corollary}
\begin{proof} This follows immediately from Lemma \ref{lem:geneigtrs} and from
\eqref{eq:detMfact} in its proof.
\end{proof}
The following theorem shows that the global optimal solution of \TRS can be obtained  via computing an eigenpair of the pencil $
M(\lambda)$.
\begin{theorem}[\textdef{Eigenvalue Characterization of \TRSp}, \cite{AdachiIwataNakatsukasaTakeda:15}]\label{t6}
	Let $(x_g^{*},\lambda_g^{*})$ be a global optimal solution of
	\TRS with $||x_g^{*}||^2=\delta$. Then the multiplier $\lambda_g^{*}$ is equal to the largest real eigenvalue of $M(\lambda)$. Furthermore, 
	if $\lambda_g^{*}> -\lambda_{1}$, then $x_g^{*}$ can be obtained by
$ x_g^{*}=-\frac{\delta}{a^Ty_2}y_1,$
where $\begin{bmatrix}y_1\\y_2\end{bmatrix}\in \Rn \times \Rn$ is an eigenvector for $M(\lambda_g^{*})$ and also we have $ a^Ty_2\neq 0$.
	\qed
\end{theorem}
Theorem \ref{t6} establishes that the largest real eigenvalue of
$M(\lambda)$ is the Lagrange multiplier associated with the global
optimal solution of \TRSp. In the following theorem, we prove that if
\TRS has a \LNGMp, then the corresponding Lagrange multiplier is  the second largest real eigenvalue of $M(\lambda) $. This is the main result of this section.
\begin{theorem}[Eigenvalue Characterization of \LNGMp]
	\label{mainlngmpeig}
	Let $x^{*}$ be a \LNGMp. Then the corresponding Lagrange multiplier $\lambda^{*}$ is equal to the second largest real eigenvalue of $M(\lambda)$. Moreover, $x^{*}$ can be computed as $x^{*}=-\frac{\delta}{a^Ty_2}y_1,$
where $\begin{bmatrix}y_1\\y_2\end{bmatrix}$ is an eigenvector for $M(\lambda^{*})$ and we also have $a^Ty_2\neq 0$.
\end{theorem}
\begin{proof}
From Theorem \ref{thm:lngmmain} we have
$\lambda^* \in (\max\{0,-\lambda_2\},-\lambda_1)$. Moreover,
$||x^*||^2=\delta$ and it follows from Lemma \ref{lem:geneigtrs} that
$\lambda^*$  is an eigenvalue of $M(\lambda)$, i.e. $\det
M(\lambda^*)=0$. We know that the hard case does not hold, see Corollary
\ref{cor:hardcase}. Therefore, by Theorem \ref{t6} and the optimality
conditions in \eqref{eq:optTRS}, we get
that the largest real eigenvalue of $M(\lambda)$ is the unique
multiplier associated with  the global optimal solution of \TRS and is the
unique root of equation $\phi(\lambda)-\delta=0$  in  $(-\lambda_1, \infty)$.
Moreover, it follows from Theorem \ref{lngmone}  that $\lambda^*$, the
multiplier corresponding to the \LNGMp, is positive and is
the largest root of the equation $\phi(\lambda)-\delta=0$  in $(-\lambda_2, -\lambda_1)$. 
Next, note that Lemma \ref{lem:geneigtrs} implies that
$-\lambda_1$ is not an eigenvalue of $M(\lambda)$. 
From the interval considerations for the optimum of \TRS and \eTRSp,
this establishes 
 that $\lambda^*$ is the second largest real eigenvalue of $M(\lambda)$. 

 Now let $\begin{pmatrix} y_1\\y_2\end{pmatrix}$ be an eigenvector for 
$\lambda^*$ for $M(\lambda)$. We have
 \begin{align}
	 &(A+\lambda^* I)y_2=y_1,\label{13}\\
  &(A+\lambda^* I)y_1=\frac 1{\delta}  aa^T  y_2. \label{14}
 \end{align}
 We first show that $a^Ty_2 \neq 0$. Suppose by contradiction that
 $a^Ty_2=0$. Then, since $(A+\lambda^* I)$ is nonsingular, we obtain
 first that $y_1=0$ from the second equation which then implies $y_2=0$
 from the first equation, i.e.,~we have 
 $y_1=y_2=0$,
 a contraction of the fact that $\begin{pmatrix} y_1\\y_2\end{pmatrix}$ is an 
	 eigenvector. Hence, $a^Ty_2 \neq 0$. Thus, (\ref{14}) implies
	 that $x^*= \frac{-\delta}{a^Ty_2}y_1$ satisfies
 \begin{align}\label{phi}
  (A+\lambda^* I)x^*=-a.
  \end{align}
 Moreover, we have
 $$||x^*||^2=\frac{\delta^2}{(a^Ty_2)^2}y_1^Ty_1=\frac{\delta^2}{(a^Ty_2)^2}y_2^T(A+\lambda^*
 I)(A+\lambda^* I)^{-1}\frac{aa^T}{\delta}y_2=\delta.$$
\end{proof}

\section{Strong Duality and Stability for \eTRSp}
\label{sect:charetrs}
\subsection{Characterization of Strong Duality for \eTRSp}
\label{sect:charact}
 A necessary and sufficient condition for strong duality of the problem
 of minimizing  a quadratic function over two quadratic inequality
 constraints, when one of them is  strictly convex, is presented in
 \cite{MR2486048}. 
 Since  \eTRS is a  special case, we have the following.
 
 \begin{theorem}[Characterization Strong Duality \eTRSp]
\label{t3}
Strong duality fails for  \eTRS if, and only if, there exist multipliers 
$\lambda, \mu $ such that the following hold:
\begin{enumerate}
	\item 
		\label{item:mults}
		$\lambda>0$ and $\mu >0$;
\item  
		\label{item:semidefs}
	$A+\lambda I\succeq 0$, and $\Rank (A+\lambda I)=n-1$;
\item  
		\label{item:twosystms}
	The following system of linear equations is consistent.
	\begin{equation}
		\label{e1}
	2(A+\lambda I)x_i=-2a-\mu b, \,
		x_i^Tx_i=\delta, \,i=1,2, \qquad (b^Tx_1-\beta)(b^Tx_2-\beta)<0.
	\end{equation}
\end{enumerate}
\end{theorem}
\begin{proof}
This follows immediately from the characterization in 
{\cite[Thm 5.2]{MR2486048}} for two quadratic constraints, since the
affine constraint is a special case of a quadratic constraint.
\end{proof}
It is interesting to translate this theorem under our special
assumptions and the language of the \emph{hard case}. In fact, we see
that loss of strong duality is directly connected to the  hard case in
\TRSp. Note that the hard case is identified by obtaining a feasible
solution that satisfies all the optimality conditions except for
complementary slackness.
\begin{corollary}
	\label{cor:hardcasemu}
	Consider the \textdef{Lagrangian dual of \eTRS} in parametric
	form.
\[
	\textdef{$d^*_{e\TRS}$}:=\max_{\mu \geq 0} g(\mu),
\]
where the \textdef{dual function, $g(\mu)$}, with $\lambda$ implicit in $g$,
is a \textdef{parametric \TRSp, \TRSmup},
\begin{equation}
	\label{TRSmu}
	\tag{\TRSmu}
	 \textdef{$g(\mu)$}:=
	\max_{\lambda \geq 0} \min_x \big[L(x,\lambda)+\mu
	b^Tx\big]-\mu\beta
\end{equation}
Then strong duality fails for \eTRS if, and only if,  
there exists $\mu> 0$ such that the  parametrized \TRSmu
has a \underline{hard case} solution $x_{\mu}^*$ that satisfies all
the optimality conditions except for complementary slackness, i.e.,
\[
	\|x_{\mu}^*\|^2<\delta, \quad b^Tx_{\mu}^*=\beta.
\]
\end{corollary}
\begin{proof}
Since \eTRS is a convex problem if $\lambda_1 \geq 0$, without loss of generality
we assume that $\lambda_1<0$. We conclude that the optimal Lagrange
multiplier for \TRSmu satisfies $\lambda > 0$ and moreover there
exists an optimal solution on the boundary of the trust region.

The three conditions in Theorem \ref{t3} are equivalent to the
optimality conditions for the parametrized problem at $\mu$. And the two
points $x_i, i=1,2$ are on opposite sides of the affine manifold for the
linear constraint. We note that necessarily $0\neq 
v:= x_1-x_2\in \Null(A+\lambda I)$. Therefore $v$ is the required
eigenvector and this is equivalent to
finding the convex combination
$x^*=\alpha x_1+ (1-\alpha)x_2, \alpha \in (0,1)$ with $b^Tx^*=\beta$
and necessarily $\|x^*\|<\delta$.

Therefore, the parametrized \TRS has multiple optimal solutions and the
hard case holds for
the corresponding \TRSmup, i.e.,~$2a+\mu b\in \Range(A-\lambda_1 I),
\, \lambda^*=-\lambda_1$.

More details on $\|x_{\mu}^*\|^2<\delta$ and the relation with
the minimum norm solution
 $\hat x:=\frac 12(A-\lambda_{1}I)^\dag(-2a-\mu b)$ are discussed
in Section \ref{sect:muvalues}, where we define the
\textdef{Moore-Penrose generalized inverse, $C^\dag$}.
In fact, necessarily
$\|x_{\mu}^*\|^2= \frac 12(A-\lambda_{1}I)^\dag(-2a-\mu b) +v$ for $v\in \Null
(A-\lambda_{1}I)$. 
\index{$C^\dag$, Moore-Penrose generalized inverse}
\end{proof}
\begin{remark}
Corollary \ref{cor:hardcasemu} illustrates the geometry of strong
duality in terms of the parametrized \TRSmup. If we start with $\mu=0$
and increase $\mu \uparrow$, then the corresponding optimal solution of
\TRSmu moves on the boundary of the trust region. If we encounter the
boundary of the linear constraint first then strong duality holds. On
the other hand if we encounter the hard case at $\mu>0$
and if we can move using the nullspace $\bar x=x_{\mu}+v$ so that
$\|\bar x \|^2<\delta, b^T\bar x =\beta$, then strong duality fails.

This means that given a \TRS we can characterize all the $b,\beta$ where
strong duality would fail using the characterization of the hard case.
\end{remark}

We know that strong duality fails if the \LNGM is the optimum for
\eTRSp.  We now see that it requires a special eigenvalue configuration for
strong duality to fail if the linear constraint is active.
\begin{theorem}
	\label{thm:strdual2eigs}
	Suppose that $x^*$ solves \eTRS with $b^Tx^*=\beta$. 
	Suppose that $\lambda_2<0$. Then strong duality holds for
	\eTRSp.
\end{theorem}
\begin{proof}
As above, we can construct a full column rank matrix $W$ to represent
$\Null(b^T)$. From interlacing of eigenvalues we get that
$\lambda_{min}(W^TAW)<0$. Therefore, there exists an optimal solution on
the boundary of the trust region for the projected problem,
i.e.,~complementary slackness holds.  This means that the optimum for
\eTRS is also on the boundary of the trust region constraint.  We can
therefore add a multiple of the identity to the Hessian of the original
problem and obtain a
convex equivalent problem. This shows that strong duality holds.
The dual problem is equivalent to perturbing the Hessian to
$Q-\lambda_1 I$ as long as we subtract the constant $\lambda_1 \beta$.
\end{proof}

\subsection{Stability for \eTRSp}
\label{sect:stablty}
We now see that the \eTRS is stable with respect to perturbations in the data.
\index{\LICQp, linear independence constraint qualification}
\begin{lemma}
Recall that we have made Assumptions \ref{assum:slater} and
\ref{assum:indefhard}.
Let $x^*$ be the optimal solution for \eTRSp.  Then the 
\textdef{linear independence constraint qualification, \LICQp}, holds at $x^*$. 
Moreover, $x^*$ is the unique optimal solution if the second 
constraint is inactive. 
Thus unique Lagrange multipliers $\lambda^*_1,\lambda^*_2$
exist for the two constraints, respectively.\footnote{We note 
that the optimum does not have to be unique for
the projected problem, i.e.,~though the hard case does not hold for
\TRSp, it can hold for the projected problem.}
\end{lemma}
\begin{proof}
Suppose that the second constraint is inactive $b^Tx^*< \beta$.
First we note that the first constraint is active by the $\lambda_1(A)<0$
assumption and the gradients of the active constraint is $2x^* \neq 0$.
Therefore the \LICQ holds. This immediately implies that $\lambda^*_1
>0$ exists. Moreover, both $x^*$ and the optimal Lagrange multiplier 
$\lambda^*$ are unique by the $\lambda_1(A)<\lambda_2(A)$ assumption.

If the second constraint is active $b^Tx^*= \beta$, then
$x^*$ is the optimal solution of the projected problem. If the first
constraint is inactive, we are done as $\{b\}$ is a linearly independent
set.  And it is clear from the geometry that if both constraints are
active then the gradients $\{b,2x^*\}$ are linearly dependent 
only if strict feasibility fails, a contradiction. Therefore, \LICQ
holds and the multipliers are unique. 
\end{proof}
\begin{corollary}
	The \eTRS is a stable problem with respect to perturbations in
	the data.  
\end{corollary}
\begin{proof}
	This follows from standard results in sensitivity analysis since
	the Lagrange multipliers are unique, satisfy LICQ and the
	feasible set is compact, e.g.,~\cite{Fia:83}.
\end{proof}
\begin{remark}
We note that these results on stability along with standard sensitivity
results on eigenvalue algorithms imply that our approach is a robust
method for solving \eTRSp.

In addition, strict complementarity can fail for \eTRSp. If the \LNGM is
the optimal solution for \eTRSp, then one can perturb the linear
constraint till it becomes active. It is therefore a redundant
constraint illustrating that the corresponding Lagrange multiplier can
be zero. This would then be a degenerate problem and perturbing the
linear constraint further can make the projected trust region optimal
point the optimum for \eTRSp, i.e.,~the result is a jump in the optimal
solution.
\end{remark}

\section{Algorithm and Subproblems for \eTRSp}
\label{sect:globetrs}
We now describe our proposed method to solve \eTRS in 
Algorithm \ref{alg:eTRSalgorithm}.
This finds the global optimal solution for the general problem of \eTRSp.
We include the details about the global minimizer for \TRS and the
details for the subproblems that need to be solved. We do \emph{not}
assume that the global minimizer of \TRS is infeasible in the details of
our algorithm, i.e.,~our algorithm solves the general case.

\begin{lemma}\label{lem:strfaillngm}
	Strong duality fails for the \LNGMp.
\end{lemma}
\begin{proof}
The Lagrangian of \TRS is given in \eqref{eq:lagrtrs}.
The Lagrangian dual of \TRS is
$\max_{\lambda \geq 0} \min_x L(x,\lambda)$. Since the inner problem is
a minimization of a quadratic, for it to be finite
we get the necessary (hidden) condition
that the Hessian of the quadratic $A+\lambda I \succeq 0$. This
contradicts the Lagrange multiplier condition for \LNGM given in Theorem
\ref{lngmone}.
\end{proof}
\begin{theorem}\label{t4}
Suppose that strong duality holds for \eTRS and that the optimal
solution of \eTRS is $x^{*}$. 
Then 
$b^Tx^{*}=\beta$  and $x^{*}$ is a global optimal solution of \TRS
after projection onto the linear manifold of the linear constraint.
\end{theorem}
\begin{proof}
If $b^Tx^{*}<\beta$, then either $x^*$ is the global minimizer or a
\LNGMp. Since strong duality fails for the \LNGMp, we conclude that it
must be the global minimizer of the \TRSp. But our Assumption \ref{assum:slater}
means that the global minimizer is infeasible for \eTRSp.

If the linear inequality is active, then we have a \TRS problem after a
projection onto the linear manifold and we obtain the global minimizer on
this affine manifold.
\end{proof}

\subsection{Main Algorithm}
 Theorem \ref{t4} suggests the following Algorithm \ref{alg:eTRSalgorithm} 
 for  \eTRSp. Without loss of generality, by Lemma \ref{lem:slater}, we
 can assume that strict feasibility holds.

 In addition, we see that the cost of the algorithm in the worst case is
  to find $\lambda_1, \lambda_2$ and eigenvector
$v_1$ for $\lambda_1$; check for strong duality; find the \TRS and
projected \TRS optima or the \LNGM and the projected \TRS optima.

\begin{table}[ht]
 \begin{mdframed}
%\begin{scriptsize}
%\begin{center}
	 \begin{algorithm}
	 \label{alg:eTRSalgorithm}
	 \begin{algorithmic}
 \STATE 
 ~~\\
 ~~\\{\bf INPUT:}  $A\in \Sn, a,b\in \Rn, \delta \in \R_{++}, \beta \in
 \R$ with
	$-\sqrt \delta \|b\| < \beta$. 
~~\\{\bf INITIALIZATION:}  Solve the symmetric eigenvalue problem
for $\lambda_1, \lambda_2$ and eigenvector $v_1$ for $\lambda_1$.
		 {\setstretch{.5}
\begin{enumerate}
	\item[{\bf IF:}]
$\lambda_1\geq 0$ or $\lambda_1=\lambda_2$, 
	{\bf THEN} Strong duality holds;
          solve \TRS for $x$.
\begin{enumerate}
	\item[{\bf IF:}]
          $x$ is feasible,  {\bf THEN} it is opt. STOP.
	\item[{\bf ELSE:}]
          Solve the projected \TRS problem for $x$; it is opt. STOP.
	\item[{\bf END:}]
\end{enumerate}
	\item[{\bf ELSE:}]
       Check the strong duality condition for \eTRSp.
       \begin{enumerate}
	\item[{\bf IF:}]
       strong duality holds, {\bf THEN} solve \TRS for $x$.
       \begin{enumerate}
	\item[{\bf IF:}] $x$ is feasible, {\bf THEN} it is opt. STOP.
	\item[{\bf ELSE:}]
          Solve the projected \TRS problem for x; it is opt.  STOP.
	\item[{\bf END:}]
\end{enumerate}
	\item[{\bf ELSE:}]
          Solve for the projected \TRS  and the \LNGM if it exists;
	  discard \LNGM if it is not feasible;
	  choose the $x$ as the best of the remaining solutions; it is opt. STOP.
	\item[{\bf END:}]
\end{enumerate}
	\item[{\bf END:}]
\end{enumerate}
}                    % for setstretch
 {\bf OUTPUT:} x is optimizer of \eTRSp. 
	 \end{algorithmic}
 \end{algorithm}
 \end{mdframed}
%\end{center}
%\end{scriptsize}
 \caption{\scriptsize Algorithm: Solve the General (strictly feasible) \eTRSp}
\end{table}

Recall that, if the \LNGM exists then we can use Theorem \ref{mainlngmpeig} 
and find it efficiently via the
second largest real eigenvalue of the matrix pencil.
The other subproblems are now discussed.

\subsection{Subproblems}

\subsubsection{Verifying Strong Duality}
	\label{sect:muvalues}

 To specify the value of $\mu$ in Theorem \ref{t3}, 
 first notice that, for given $\mu$, system (\ref{e1})
 is consistent if, and only if, $v^T(2a+\mu b)=0$ where $v$ is a 
 normalized eigenvector for $\lambda_{1}$. Next, let us 
 consider the following cases:
 \begin{enumerate}
	 \item \textbf{$v^Tb=0$:}
 In this case, we show that strong duality holds for \eTRSp. We show
 this by contradiction.  Suppose that strong duality does not
hold for \eTRSp. Then system (\ref{e1}) has two solutions $x_1$ and
$x_2$ satisfying $x_i^Tx_i=\delta $, $i=1,2$, and
$(b^Tx_1-\beta)(b^Tx_2-\beta)<0$ for some $\mu>0$.
 Moreover,  we know that the solutions $x_1$ and $x_2$  necessarily have the form
 $x_1= \frac 12(A-\lambda_{1}I)^\dag(-2a-\mu b)+z_1$ and
 $x_2=\frac 12(A-\lambda_{1}I)^\dag (-2a-\mu b)+z_2$ where $z_i$, for
$i=1,2$,   is  an eigenvector corresponding to
$\lambda_{1}$. By the fact that $b$ is orthogonal to
the eigenspace of $\lambda_{1}$ ($\lambda_{1}$ has
multiplicity one), we have  $b^Tx_1-\beta=b^Tx_2-\beta$, a contradiction to the
fact that $(b^Tx_1-\beta)(b^Tx_2-\beta)<0$, i.e.,~we have strong
duality  for \eTRSp.
\item
 \textbf{$ v^Tb\neq 0$:}
In this case, consistency of system (\ref{e1}), i.e., $v^T(2a+ \mu b)=0$
implies that necessarily $\mu=\frac{-2v^Ta}{v^Tb}$. If $\mu = 0$, it
follows from  Theorem \ref{t3} that \eTRS enjoys  strong duality. If
$\mu>0$, then  strong duality does not hold for \eTRS if, and only if,
system (\ref{e1}) for $\mu=\frac{-2v^Ta}{v^Tb}$ has two solutions $x_1$
and $x_2$ satisfying $x_i^Tx_i=\delta $, $i=1,2$, and
$(b^Tx_1-\beta)(b^Tx_2-\beta)<0$.
\end{enumerate}

To verify whether strong duality holds 
 we suppose that $x_i$, $i=1,2$  are as defined in Theorem \ref{t3}.
clearly,
 $x_i=x_p+\alpha_i v$  where  $v$ is a normalized eigenvector associated
 with $\lambda_{1}$, $x_p=\frac 12(A-\lambda_{1} I)^{\dag}(-2a-\mu b)$ and $\alpha_i$, $i=1,2$, are  roots of the following quadratic equation.
 $$\alpha^2+2\alpha v^Tx_p+x_p^Tx_p-\delta =0.$$
The main task in finding $x_i$, $i=1,2$, is computing $x_p$. In the
sequel, we show that $x_p$ is indeed the solution of a symmetric
positive definite linear system. To see this,  let us consider the
eigenvalue decomposition of $A$ defined as before in which $Q$ contains
$v$ as its first column. Noting that $v^T(2a+\mu b)=0$, we have
\begin{align*}
	(A+\gamma vv^T-\lambda_{1}I)^{-1}(-2a-\mu b)&=Q(\Lambda+\gamma
	e_1e_1^T-\lambda_{1} I)^{-1}Q^T(-2a-\mu
	b)\\&=Q(\Lambda-\lambda_{1} I)^{\dag}Q^T(-2a-\mu b)\\&=(A
	-\lambda_{1}I)^{\dag}(-2a-\mu b),
\end{align*}
where $\gamma$ is a positive constant and $e_1$ is the first unit vector. This implies that $x_p$ can be computed efficiently by applying the conjugate gradient algorithm to the following  positive definite  system.
\begin{align*}
	2(A+\gamma vv^T-\lambda_{1} I)x_p=(-2a-\mu b).
 \end{align*}
 However, we note that the perturbation with $\gamma vv^T$ is \emph{not}
 required since the right-hand side $(-2a-\mu b) \in \Range(A-\lambda_1 I)$.
 The MATLAB \emph{pcg} works fine even though the matrix is singular.

 \subsubsection{Solving the \TRS Subproblem}
The main work of the algorithms lie in solving generalized eigenvalue problems.
For the \TRSp, we use the method of \cite{AdachiIwataNakatsukasaTakeda:15}
 that solves the scaled \TRS
 \begin{align}\label{cl}
\min \quad &\frac{1}{2}x^TAx +a^Tx \notag \\
&x^TBx\leq \delta,
\end{align}
where $B$ is a positive definite matrix. The algorithm computes one 
generalized eigenpair and 
is able to handle  the hard case efficiently. Specifically, it is shown that 
the optimal Lagrange multiplier corresponding to the  solution of (\ref{cl}) 
is the largest real eigenvalue of  the $2n\times 2n$ matrix pencil 
$M_0+\lambda M_1$, where
\begin{align*}
	\tilde M(\lambda)=M_0+\lambda M_1, \qquad
M_0=\begin{bmatrix}
-B& A\\
A&-\frac{aa^T}{\delta}
\end{bmatrix}, \,\, M_1=\begin{bmatrix}
O_{n\times n}& B\\B& O_{n\times n}
\end{bmatrix}.
\end{align*}
As above we have an equivalent result to Lemma \ref{lem:geneigtrs} that every
nonzero KKT multiplier is a generalized eigenvalue of the pencil, $\det
(\tilde M(\lambda))=0$.
\begin{lemma}[\textdef{Generalized Eigenvalue of Pencil}, {\cite[Lemma
	3.1]{AdachiIwataNakatsukasaTakeda:15}}]
	\label{lem:geneigtrsB}
	For every nonzero KKT multiplier
	$\lambda_g^*\neq 0$ for \eqref{cl}
	with equality in the quadratic constraint
we have $\det \tilde M(\lambda_g^*)=0$, i.e.,~$\lambda^*_g$
is a {generalized eigenvalue of the pencil $\tilde M(\lambda)$}.
	\qed
\end{lemma}

\begin{table}[ht]
 \begin{mdframed}
	 \begin{algorithm}
	\label{alg:trs}
\begin{enumerate}
\item
Solve $Ax_0=-a$ by the conjugate gradient algorithm and keep $x_0$ if it
is feasible, i.e.,~if $x_0^TBx_0\leq \delta$.
\item
	\label{item:geneig}
Compute $\lambda^{*}_g$, the largest generalized eigenvalue of the 
symmetric regular pencil $M_0+\lambda M_1$, and a corresponding eigenvector 
$\begin{pmatrix} y_1\\y_2\end{pmatrix}$, i.e.,
\begin{align}
\begin{bmatrix}
-B& A\\A & -\frac{aa^T}{\delta}
\end{bmatrix}
\begin{pmatrix}y_1\\y_2
\end{pmatrix}=-\lambda^{*}_g
\begin{bmatrix}
O_{n\times n}& B\\
B& O_{n\times n }
 \end{bmatrix}\begin{pmatrix}y_1\\y_2\end{pmatrix}.
\end{align}
\item
If $||y_1||\leq \tau $ (default is $\tau=10^{-4}$), then the hard case
is detected; run Steps \ref{item:compH} to \ref{item:eigenv}. Else go to Step
\ref{item:x1}.
	
\item 
	\label{item:compH}
	Compute $H:= (A+\lambda^{*}_g B+\alpha
	\sum_{i=1}^dBv_iv_i^TB)$ where $V=[v_1, ..., v_d]$ is a basis of
	$\Null(A+\lambda^{*}_g B)$ that is $B$-orthogonal, i.e.,
	$V^TBV=I$, $d=\dim (\Null(A+\lambda^{*}_g B))$ and $\alpha $ is an arbitrary positive scalar.
\item
	\label{item:conggrad}
Solve $Hq=-a$ by the conjugate gradient algorithm.
\item
	\label{item:eigenv}
Take an eigenvector $v$ computed above, and find $\eta$ such that $(q+\eta v)^TB(q+\eta v)=\delta$ and return $x^{*}=q+\eta v$ as global optimal solution of (\ref{cl}).
\item
	\label{item:x1}
Set $x_1=-\text{sign}(a^Ty_2)\sqrt{\delta}\frac{y_1}{\sqrt{y_1^TBy_1}}$.
\item
The global optimal solution of (\ref{cl}) is either $x_1$ or $x_0$, whichever gives the smaller objective value.
\end{enumerate}
\end{algorithm}
 \end{mdframed}
 \caption{\scriptsize Algorithm: Solve scaled \TRS \eqref{cl}, {\cite[Theorem 3.1]{AdachiIwataNakatsukasaTakeda:15}}}
\end{table}

 \subsubsection{Solving the Projected \TRS Subproblem}
We can eliminate the equality constraint $b^Tx=\beta$
to solve the projected \TRSp.  For ease of exposition, we assume that
\[
	|b_1| \geq |b_2| \geq \ldots \geq |b_r| > 0= b_{r+1}=\ldots = b_n.
\]
In order to find a basis of $\Null(b^T)$, we define \textdef{$\bar b$}$:=
\begin{pmatrix} b_2^{-1}&\ldots &b_r^{-1}  \end{pmatrix}^T$ 
and the matrix
\[
	\textdef{$W:=$}
	\left[
	\begin{array}{c|ccc}
		-b_1^{-1}e_{r-1}^T & 0_{n-r} \\
		\hline
		\Diag(\bar b) & 0\\0 & I_{n-r}
	\end{array}
\right] \in \R^{n\times(n-1)}.
\]
\index{$\hat x$, particular solution}
Define a \textdef{particular solution, $\hat x$} satisfying $b^T\hat x =
\beta, \|\hat x\|^2<\delta$.\footnote{Some scaling issues can arise here. It is preferable to
	take $\hat x$ strictly feasible for the trust region
constraint.} We choose
\begin{equation}
	\label{eq:xhat}
	\hat x=\left\{\begin{array}{cc}
			0, & \text{if   }  \beta = 0 \\
			\frac \beta {\|b\|^2} b, & \text {if   } \beta
			\neq 0.
		\end{array}\right.
\footnote{We note that the choice $\hat x=0$ simplifies the nonhomogeneous \nTRS
below.}
\end{equation}
Then it is clear that
\[
b^Tx = \beta \iff x = \hat x + Wy, \text{  for some  } y\in \R^{n-1}.
\]
We can now substitute for $x$ into \eTRS and eliminate the linear
equality constraint. The objective function becomes
\[
	(\hat x + Wy)^TA (\hat x + Wy)+2a^T(\hat x + Wy)=
	\left[y^T(W^TAW)y + 2\left(W^T(a+A\hat x)\right)^Ty\right] + 
	\left[ (A \hat x +2a)^T\hat x\right].
\]
The constraint becomes
\[
		y^T(W^TW)y + 2(W^T\hat x)^Ty\leq \delta -{\hat x}^T  \hat x.
\]
We get the following equivalent problem in the case that the linear constraint 
is active. 
\begin{equation}
	\label{TRSproj}
	\tag{\TRSproj}
	\index{\TRSprojp, projected \eTRSp}
	\begin{array}{rcl}
		&\min  & y^T(W^TAW)y + 2(W^T(a+A\hat x))^Ty \\
	& \text{s.t.} & y^T(W^TW)y + 2(W^T\hat x)^Ty\leq \delta -{\hat
	x}^T  \hat x
\end{array}
\end{equation}
We let
\[
B:=W^TW, \, \hat A:=W^TAW, \, \hat a:= W^T(a+A\hat x), \, \hat b:= 2(W^T\hat x).,\  \hat \delta= \delta-{\hat x}^T \hat x.
\]
Therefore, we need to solve the \textdef{nonhomogeneous \TRSp, \nTRSp}
\index{\nTRSp, nonhomogeneous \TRSp}
\begin{equation}
	\label{nTRS}
	\tag{\nTRS}
	\begin{array}{rcl}
		&\min  & x^T\hat Ax + 2\hat a^Tx \\
	 & \text{s.t.} & x^TBx + \hat b^Tx\leq \hat\delta.
\end{array}
		\footnote{We note
		again here that if $\beta =  0$ then we can choose
	$\hat x=0$ and the homogeneous \TRS is maintained.}
\end{equation}

By the  change of variables
\[
	x\leftarrow y+g, \quad \text{ with } \quad  2Bg=-\hat b,
\]
we get
\[
	\begin{array}{rcl}
x^T\hat A x + 2 \hat a^Tx 
&=&
(y+g)^T\hat A (y+g) + 2 \hat a^T(y+g) 
\\&=&
y^T\hat A y + 2 (\hat A g+ \hat a)^Ty  +\text{  constant}.
	\end{array}
\]
and
\[
	\begin{array}{rcl}
x^T B x +  \hat b^Tx 
&=&
(y+g)^T B (y+g) +  \hat b^T(y+g) 
\\&=&
y^T B y +  (2 B g+ \hat b)^Ty  +g^TBg +b^Tg
\\&=&
y^T B y +  g^TBg +b^Tg
	\end{array}
\]
We write \nTRS as the
\textdef{scaled homogeneous \TRSp, \sTRSp},
\index{\sTRSp, scaled homogeneous \TRSp}
\begin{equation}
	\label{sTRSh}
	\tag{\sTRS}
	\begin{array}{rcl}
		&\min  & y^T\hat Ay + 2(\hat Ag+\hat a)^Ty \\
	 & \text{s.t.} & y^TBy \leq \hat\delta-g^TBg-{\hat b}^Tg.
\end{array}
\end{equation}
This means we can directly apply the approach in
\cite{AdachiIwataNakatsukasaTakeda:15} where the scaled \TRS is solved
using the generalized eigenvalue approach.
\begin{remark}
When we solve for the optimimum in \sTRS using \eqref{alg:trs}
we do not form $B$ explicitly but
exploit the rank one update structure of $W$ and its inverse. This means
we can exploit the original sparsity in $A$ in the objective function
and in the, now scaled, $I$ in the original trust region constraint 
when performing the
matrix-vector multiplications needed for \emph{eigs} in MATLAB.
Let 
\[
	\bar B := \Diag(\bar b), \qquad
	\bar e := \begin{pmatrix}
	e_{r-1}^T \cr  
		\hline
	0_{n-r}
	\end{pmatrix}.
\]
Note that 
	\[
		\begin{array}{rcl}
		B
		&=&
	\left[
	\begin{array}{c|ccc}
		\bar B^2	&0 \cr
		\hline
		  0 & I_{n-r}
	\end{array}
\right] +
	b_1^{-2}
	\bar e \bar e^T
		\\&=&
	\left\{	\left[
	\begin{array}{c|ccc}
		\bar B	&0 \cr
		\hline
		  0 & I_{n-r}
	\end{array}
\right] +
	ww^T
\right\}
	\left\{	\left[
	\begin{array}{c|ccc}
		\bar B	&0 \cr
		\hline
		  0 & I_{n-r}
	\end{array}
\right] +
	ww^T
\right\}
		\\&=&
B^{1/2} B^{1/2}.
	\end{array}
	\]
We can then find the appropriate rank one update of 
	$\left[
	\begin{array}{c|ccc}
		\bar B	&0 \cr
		\hline
		  0 & I_{n-r}
	\end{array}
\right]$ to find the inverse $B^{-1/2}$. Therefore we can take a
diagonal congruence of both sides of \eqref{alg:trs} and obtain a simple
right-hand side of the generalized eigenvalue problem.
\end{remark}

\section{Numerical Results}
\label{sect:nums}
\index{\Socsdpp, second order cone and semidefinite programming}
We now present our numerical results to illustrate
the efficiency of the new algorithm. We compare with  the 
second order cone and semidefinite programming, \Socsdpp,
reformulation in \cite{BuAn:11} on some small instances 
as this reformulation is not able to handle large instances.  Hence, for large 
instances we just report the solution obtained by our new algorithm.  

All computations were done  in MATLAB 8.6.0.267246 (R2015b) on a 
Dell Optiplex 9020 with 16GB RAM with Windows 7.
To solve the \Socsdp reformulation, we used SeDuMi 1.3, \cite{S98guide}.

\subsection{Four Classes of Test Problems}
We divide our tests into \emph{four classes} I,II,III,IV, of test problems.
\subsubsection{Class I}
In this section, we apply our algorithm and the \Socsdp
reformulation to some \eTRS instances  for which the \LNGM of the
corresponding  \TRS is a good candidate for the 
global optimal solution of \eTRSp. To
generate the desirable random instances of \eTRSp, we proceed as follows. First
we construct a \TRS problem having a local non-global minimizer based
on Theorem \ref{mainlngmpeig}.
Then we add the inequality constraint $b^Tx\leq \beta$ to enforce that
 the global minimizer of \TRS is infeasible but that the \LNGM
 remains feasible.

Comparison with  the \Socsdp reformulation is given on some  small
instances in Table \ref{table:randomrepeat}. 
We follow~\cite{AdachiIwataNakatsukasaTakeda:15} and report
the relative objective function difference
$$\frac{|q(x^{*})-q(x_{best})|}{|q(x_{best})|} \qquad \text{accuracy
measure},$$
where $x^{*}$ is the computed solution by each method and $x_{best}$ is
the solution with the smallest objective value among the two algorithms. For
each dimension, we have generated $10$  \eTRS instances. We report the
dimension $n$, and the average values of the relative accuracy,
the run  time in cpu-seconds and we include the time
taken for checking the strong duality property of \eTRS in Algorithm  
%\ref{alg:maingeneraletrs}.
\ref{alg:eTRSalgorithm}.
Moreover, for each dimension, \# \LNGM denotes the number
of test problems among the $10$ instances for which our algorithm has detected
the \LNGM of the corresponding \TRS as a global optimal  solution of
\eTRSp.  It should be noted that the algorithm which gets $x_{best}$
varies from problem to problem and since we are reporting  the average
of $10$ runs, we can have a positive accuracy in both columns of the table.

%% old table 1 ... not needed ?????
%%%  TABLE 1
%%\begin{table}[ht]
%%\begin{scriptsize}
%%\begin{center}
%%\begin{tabular}{|l| l| c| c  |c| c| c|c| }
%%	\hline
%%	& \multicolumn{2}{|c|}{Accuracy}
%%& \multicolumn{2}{|c|}{Time algor./strong dual.}
%%	& \multicolumn{2}{|c|}{\# instances \LNGMp}\\
%%	\hline
%%	&Main Algor. &\Socsdp
%%	& Main Algor. & \Socsdp
%%	& Main Algor. & \Socsdp\\
%%\hline \hline
%%$n=100$& $0$&$1.6440\times 10^{-10}$ &$0.1289/0.0391$&$1.4511$&$10$&$0$\\
%%\hline
%%$n=200$& $0$&$4.0719\times 10^{-10}$ &0.1535/0.0480&10.1081&10&0\\
%%\hline
%%$n=300$& $0$&$1.4012\times 10^{-10}$ &$0.1882/0.0431$&$37.18$&$10$&$0$\\
%%\hline
%%$n=400$& $0$&$1.4012\times 10^{-10}$ &$0.3572/0.0481$&$107.76$&$10$&$0$\\
%%\hline \hline
%%\end{tabular}
%%\caption{ Comparison with  \Socsdp reformulation; random
%%problems}
%%\label{table:randomrepeat}
%%\end{center}
%%\end{scriptsize}
%%\end{table}

%  TABLE 1
\begin{table}[ht]
\begin{scriptsize}
\begin{center}
	\input{TableTempT1};
	\caption{\scriptsize Class I: Comparison with  \Socsdp reformulation.}
\label{table:randomrepeat}
\end{center}
\end{scriptsize}
\end{table}

We see in Table \ref{table:randomrepeat} that our algorithm finds the global
optimal solution of \eTRS significantly faster than the \Socsdp
reformulation and with improved accuracy. The generated matrix $A$ in
this the first class of test problems is dense and so we do
not perform tests of large size as the aim of our method 
is solving large sparse \eTRS instances. 

\subsubsection{Class II}
In this section we test our algorithm  on both small and large sparse
\eTRS instances.  we take advantage of the following lemma from 
\cite{Mar:94} to generate  such  \eTRS instances.
\begin{lemma}[Lemma 3.4 of \cite{Mar:94}]
	\label{test}
	Consider the \TRS problem. Suppose that
	$\lambda_1<0$,  has multiplicity one,  and the \TRS
	is an easy case instance. Then there exists $\delta_0>0$ such
	that \TRS admits a local non-global minimizer for all $\delta>\delta_0$.
	\qed
\end{lemma}

The second class of test problems are generated as follows. We
generate a random sparse symmetric matrix $A$ via
\verb=A=\verb===\verb=sprandsym(n,density)=. Next we generate the vector
$a$ via \verb=a=\verb===\verb=randn(n,1)= and make sure that $v^Ta\neq
0$ where  $v$ is the eigenvector corresponding to
$\lambda_{1}$, i.e.,~we get the  easy case \TRSp. Then we
set $\delta=4000$ following  Lemma \ref{test}. Finally we set $b=0.9 \text{xopt}$
and $c=||b||^2$ to cut off $\text{xopt}$, the global optimal solution of the
generated \TRS instance.  We have compared our algorithm with the
\Socsdp reformulation on the test problems of small size  in both
runtime and solution accuracy.  For each dimension, we have generated
$10$   \eTRS instances and
the corresponding numerical results are presented in Table
\ref{table:density}, where we report the
     dimension of the problem $n$, the algorithm run  time and the time
     taken for checking the strong duality property of \eTRS,
     %by Algorithm  %    \ref{alg:maingeneraletrs} \ref{alg:eTRSalgorithm}.
     and the accuracy
      at termination averaged over the $10$  random instances. Moreover,
      for each dimension, \# \LNGM denotes the number of test problems
      among $10$  instances for which our algorithm has detected the \LNGM
      of the corresponding \TRS as a global optimal  solution of \eTRSp.
      It should be noted that the algorithm which gets $x_{best}$ varies
      from problem to problem and since we are reporting  the average of
      $10$  runs, we have positive accuracy in the Table. Furthermore, we
      verified that in all cases, there was a positive duality gap for
      generated \eTRS instances. As in the previous test
      problems the new algorithm finds higher
      accuracy solutions in significantly shorter time than the 
      \Socsdp reformulation.

%%% old table2 ????not needed ????
%%%%%  TABLE 2
%%      \begin{table}[ht]
%%	      \begin{scriptsize}
%%\begin{center}
%%\begin{tabular}{|l| l| c| c  |c| c| c|c| }
%%	\hline
%%	& \multicolumn{2}{|c|}{Accuracy}
%%& \multicolumn{2}{|c|}{Time algor./strong dual.}
%%	& \multicolumn{2}{|c|}{\# instances \LNGMp}\\
%%	\hline
%%	&Main Algor. &\Socsdp
%%	& Main Algor. & \Socsdp
%%	& Main Algor. & \Socsdp\\
%%\hline \hline
%%$n=100$& $2.2013\times 10^{12}$&$1.0331\times 10^{-9}$ &$0.2144/0.0331$&$2.5498$&$10$&$0$\\
%%\hline
%%$n=200$& $1.4102\times 10^{-12}$&$2.0415\times 10^{-9}$ &0.3121/0.0450&14.8524&5&0\\
%%\hline
%%$n=300$& $3.7106\times 10^{-12}$&$4.2714\times 10^{-9}$ &$0.4201/0.0540$&$591634$&$5$&$0$\\
%%\hline
%%$n=400$& $4.6123\times 10^{12}$&$5.7421\times 10^{-9}$ &$0.5591/0.0671$&$164.86$&$6$&$0$\\
%%\hline \hline
%%\end{tabular}
%%\caption{Comparison  with \Socsdp reformulation; random
%%problems with density $0.1$.
%%?????has to be redone with code\_table2.m
%%and ???headings as in taable for table1.m
%%}
%%\label{table:density}
%%\end{center}
%%	      \end{scriptsize}
%%\end{table}
%%

%  TABLE 2
\begin{table}[ht]
\begin{scriptsize}
\begin{center}
	\input{TableTempT2};
\caption{\scriptsize Class II: Comparison with  \Socsdp reformulation; density
	$0.1$}
\label{table:density}
\end{center}
\end{scriptsize}
\end{table}

Now we turn to solving large sparse \eTRS instances. For this class we
just report the  results of our algorithm since the \Socsdp approach could
not handle problems of this size. Let $x^{*}$ be a global
optimal solution of \eTRS and $\lambda^{*}$ the  corresponding Lagrange 
multiplier. Depending on the context of the linear constraint being
not active or being active, we denote the error in the stationarity
condition by:
\textdef{$\KKTo$}$:=||(A+\lambda^{*} I)x^{*}+a||_{\infty}$ or
the corresponding conditions for the scaled active case, respectively;
%\textdef{$\KKTo$}$:=  
%||(\bar{A}+\lambda^{*}\bar{B})x^{*}-\bar{A}\bar{B}^{-1}g_1+g)||_{\infty}$;
and the error in complementary slackness by
\textdef{$\KKTt$}$:=\lambda^{*}(||x^{*}||^2-\delta)$ 
or the corresponding condition for the scaled linear active case,
respectively.
%\textdef{$\KKTt$}$:=\lambda^{*}(x^{*^T}\bar{B}x^{*}-g_1^T\bar{B}^{-1}g_1+\gamma_1)$.
 For each dimension, we have generated $10$  \eTRS instances.
In both cases the global
optimal solution of \eTRS is obtained from solving generalized 
eigenvalue problems. Numerical results are presented in 
Table \ref{table:sparse1}. 

%  TABLE 3
\begin{table}[ht]
\begin{scriptsize}
\begin{center}
	\input{TableTempT3};
\caption{\scriptsize Class II: Large instances; density
	$0.0001$}
\label{table:sparse1}
\end{center}
\end{scriptsize}
\end{table}

The following lemma is useful in generating test problems for the next
two classes.
\begin{lemma}[Generating \LNGMp]
	\label{lem:genlngm}
Let $A\in \Sn$ and suppose that
$\lambda_1<\min\{0,\lambda_2\}$. Then there exists linear term $a$ for
which the eigenvector associated with $\lambda_1$ is the \LNGMp.
\end{lemma}
\begin{proof}
Let $\mu\in(\max{\{0,-\lambda_2\}},-\lambda_1)$. Set $a=-(A+\mu I_n)v_1$
where $v_1$ is the eigenvector for $\lambda_1$ with $||v_1||^2=\delta$.
Then for this choice we have the first order stationary conditions.
Now let $\Range(W)=\text{Null}(v_1^T)$. Then $W^T(A+\mu I_n)W=\text{diag}
(\lambda_2+\mu, ..., \lambda_n+\mu)$. Due to the choice of $\mu$,
the diagonal matrix has all diagonal elements positive. Thus we have the
positive definiteness of the reduced Hessian. This implies that
$v_1$ is the \LNGMp.
\end{proof}

\subsubsection{Class III}
\label{sect:third}

In this section, we consider a class of large sparse \eTRS instances
for which strong Lagrangian duality holds while the corresponding \TRS
has a \LNGM which is feasible for \eTRSp.  We generate the \TRS using
the previous Lemma \ref{lem:genlngm}
and set $b=(A-\lambda_{1}I)x$ where
\verb=x=\verb===\verb=rand(n,1)=. This means that $b^Tv_1=0$  implying that
we have strong duality property for generated \eTRS instances.

Now let $x^{*}$ be a global optimal solution of \eTRSp. Then either
$b^Tx^{*}<\beta$ or $b^Tx^{*}=\beta$. Since strong duality holds, in the former
case, $x^{*}$ is the global minimizer of the corresponding \TRSp.
We define $\KKTo$ and $\KKTt$ as the previous
section. For each dimension, we have generated $10$  \eTRS instances and
the corresponding numerical results are presented in Table
\ref{table:strdualsparse}.

%%% old table 4 ... not needed any more ?????
%%     %Table 4
%%\begin{table}[ht]
%%\begin{center}
%%\begin{tabular}{|l| l| c| c|    }
%%	\hline
%%&\KKTo  &\KKTt& Time (s) \\
%%\hline \hline
%% n=10000& $5.7718\times 10^{-10}$&$-4.7549\times 10^{-11}$&0.7715(0.2220)\\
%%\hline
%% n=20000& $2.5132\times 10^{-11}$&$4.728\times 10^{-12}$&1.6168(0.5611)\\
%%\hline
%% n=40000&$3.6206\times 10^{-9}$&$3.2070\times 10^{-10}$&4.9179 (1.4789)\\
%%\hline
%% n=60000& $7.2601\times 10^{-11}$ &$2.5833\times 10^{-12}$&8.2120(2.8921)\\
%%\hline
%% n=80000& $7.7477\times 10^{-8}$ &$7.7423\times 10^{-10}$&21.7841(4.6738)\\
%% \hline
%% n=100000& $1.0961\times 10^{-9}$ &$7.7980\times 10^{-10}$&22.6937 (5.9546)\\
%%\hline \hline
%%\end{tabular}
%%\caption{Third class of test problems  with
%%strong duality and density $.0001$}
%%\label{table:strdualsparse}
%%\end{center}
%%\end{table}
%%
\begin{table}[ht]
\begin{scriptsize}
%\begin{table}[H]
\begin{center}
	\input{TableTempT4}
\caption{\scriptsize Class III:  density $0.0001$}
\label{table:strdualsparse}
\end{center}
\end{scriptsize}
\end{table}

\subsubsection{Class IV}
For this class also we follow the above Lemma \ref{lem:genlngm} to
generate \TRS having \LNGMp. 
We follow the same procedure as in Section \ref{sect:third} to obtain
$A$, $a$, $\delta$ and \LNGM but we set $b=xopt-x_l$ and
$\beta=b^T(0.9x_l+0.1xopt)$  to cut off $xopt$ but leave $x_l$ feasible
where $xopt$ and $x_l$ are the global optimal solution  and \LNGM of the
corresponding \TRSp, respectively.

%%  table 5 in paper
\begin{table}[ht]
\begin{scriptsize}
%\begin{table}[H]
\begin{center}
	\input{TableTempT6}
\caption{\scriptsize Class IV:  density $0.1$}
\label{table:randomrepeat2}
\end{center}
\end{scriptsize}
\end{table}

%%
%%  table 6 in paper
\begin{table}[ht]
\begin{scriptsize}
%\begin{table}[H]
\begin{center}
	\input{TableTempT7}
\caption{\scriptsize Class IV:  density $0.0001$}
\label{table:class4sparse}
\end{center}
\end{scriptsize}
\end{table}

\section{Conclusion}
\label{sect:concl}
In this paper we have derived a new necessary condition for the local
non-global optimal solution \LNGM of
the \TRS that is based on the second largest real generalized eigenvalue of a matrix
pencil. This is then used to derive an efficient algorithm for finding
the global minimizer of the extended \TRSp, the \eTRSp. We have presented
numerical tests to show that our method far outperforms current methods
for \eTRSp. And our method solves large sparse problems which are
too large for current methods to be applied.
We have included discussions on a characterization of when strong
duality holds for \eTRS as well as details on the stability of the
problem. 

It is well known that \TRS is important for unconstrained trust
region methods, restricted Newton methods,
for unconstrained minimization; as well it is important for general 
minimization algorithms such as sequential quadratic programming
(\SQPp) methods. For \SQP methods it is customary to solve a standard quadratic
programming problem for the search direction after using something akin
to a quasi-Newton method to guarantee convexity of the objective
function. The \eTRS we have studied can be viewed as a step toward
solving a \TRS with multiple linear constraints for the search direction
in \SQP methods.

\addcontentsline{toc}{section}{Index}
%\label{ind:index}
\printindex

\addcontentsline{toc}{section}{Bibliography}
\bibliographystyle{plain}
\bibliography{.master,.edm,.psd,.bjorBOOK}

\end{document}

%% file: TableTempT1.tex
\begin{tabular}{|c|c|c|c|c|c|c|}
\hline
& Accuracy & Accuracy & CPUsec  & CPUsec & CPUsec & \# LNGM \\
\hline
& Main Algor. & SOCP/SDP & Main Algor & Str. Dual. & SOCP/SDP & Main Algor. \\
\hline
\hline
100 & 0.0  &  1.1309e-10  &     0.043  &     0.019  & 1.372e+00  &  10  \\
200 & 0.0  &  2.9945e-10  &     0.037  &     0.012  & 8.440e+00  &  10  \\
300 & 0.0  &  2.7884e-10  &     0.040  &     0.012  & 3.193e+01  &  10  \\
400 & 0.0  &  3.1309e-10  &     0.049  &     0.018  & 9.017e+01  &  10  \\
\hline
\hline
\end{tabular}

%% file: TableTempT2.tex
\begin{tabular}{|c|c|c|c|c|c|c|}
\hline
& Accuracy & Accuracy & CPUsec  & CPUsec & CPUsec & \# LNGM \\
\hline
& Main Algor. & SOCP/SDP & Main Algor & Str. Dual. & SOCP/SDP & Main Algor. \\
\hline
\hline
100 & 0.0  &  4.2588e-09  &     0.093  &     0.028  & 1.697e+00  &  9  \\
200 & 0.0  &  1.0547e-08  &     0.128  &     0.030  & 1.167e+01  &  6  \\
300 & 0.0  &  9.3557e-09  &     0.180  &     0.036  & 4.694e+01  &  7  \\
400 & 0.0  &  3.3775e-09  &     0.252  &     0.042  & 1.287e+02  &  5  \\
\hline
\hline
\end{tabular}

%% file: TableTempT3.tex
\begin{tabular}{|c|c|c|c|c|c|}
\hline
& Opt. Cond. & Opt. Cond. & CPUsec  & CPUsec & \# LNGM \\
\hline
& KKT1 eTRS & KKT2 C.S. & Algor Time  & Str. dual. Time & Main Algor. \\
\hline
\hline
10000 & 1.4085e-08  &  -1.3688e-12  &     1.087  &     0.168 & 4 \\
20000 & 1.3465e-10  &  -7.7060e-13  &     2.506  &     0.294 & 6 \\
40000 & 1.9584e-09  &  -3.8369e-14  &    10.343  &     0.963 & 2 \\
60000 & 1.9876e-10  &  1.8024e-14  &    13.694  &     1.912 & 4 \\
80000 & 1.8937e-10  &  5.3614e-13  &    26.768  &     3.253 & 5 \\
100000 & 8.5902e-11  &  2.8473e-12  &    29.225  &     5.415 & 2 \\
\hline
\hline
\end{tabular}

%% file: TableTempT4.tex
\begin{tabular}{|c|c|c|c|c|}
\hline
& Opt. Cond. & Opt. Cond. & CPUsec  & CPUsec \\
\hline
& KKT1 eTRS & KKT2 C.S. & Algor Time  & Str. dual. Time \\
\hline
\hline
10000 & 5.4076e-14  &  5.4076e-14  &     0.313  &     0.118 \\
20000 & 3.1243e-14  &  3.1243e-14  &     0.731  &     0.242 \\
40000 & 2.0866e-12  &  2.0866e-12  &     2.279  &     0.721 \\
60000 & 8.9301e-14  &  8.9301e-14  &     3.827  &     1.448 \\
80000 & 4.5073e-14  &  4.5073e-14  &     5.998  &     2.333 \\
100000 & 9.7731e-14  &  9.7731e-14  &     9.727  &     3.820 \\
\hline
\hline
\end{tabular}

%% file: TableTempT6.tex
\begin{tabular}{|c|c|c|c|c|c|c|}
\hline
& Accuracy & Accuracy & CPUsec  & CPUsec & CPUsec & \# LNGM \\
\hline
& Main Algor. & SOCP/SDP & Main Algor & Str. Dual. & SOCP/SDP & Main Algor. \\
\hline
\hline
100 &   1.8488e-10  & 0.0  &    0.132  &     0.025  & 9.170e-01  &  10  \\
200 &   2.3815e-10  & 0.0  &    0.145  &     0.025  & 7.037e+00  &  10  \\
300 &   2.1072e-10  & 0.0  &    0.230  &     0.034  & 2.926e+01  &  10  \\
400 &   2.1792e-10  & 0.0  &    0.386  &     0.041  & 8.877e+01  &  10  \\
\hline
\hline
\end{tabular}

%% file: TableTempT7.tex
\begin{tabular}{|c|c|c|c|c|c|}
\hline
& Opt. Cond. & Opt. Cond. & CPUsec  & CPUsec & \# LNGM \\
\hline
& KKT1 eTRS & KKT2 C.S. & Algor Time  & Str. dual. Time & Main Algor. \\
\hline
\hline
10000 & 1.3807e-13  &  2.1481e-18  &     2.564  &     0.167 & 10 \\
20000 & 3.3108e-14  &  1.2592e-16  &     2.712  &     0.314 & 10 \\
40000 & 1.9213e-13  &  -9.3530e-16  &    10.682  &     0.981 & 10 \\
60000 & 3.8501e-13  &  7.6124e-16  &    19.285  &     2.060 & 10 \\
80000 & 6.2677e-14  &  3.1855e-16  &    29.587  &     3.736 & 10 \\
100000 & 1.1080e-13  &  -7.4408e-16  &    44.761  &     6.171 & 10 \\
\hline
\hline
\end{tabular}

%% file: etrscurrent.bbl
\def\cprime{$'$} \def\cprime{$'$} \def\cprime{$'$}
  \def\udot#1{\ifmmode\oalign{$#1$\crcr\hidewidth.\hidewidth
  }\else\oalign{#1\crcr\hidewidth.\hidewidth}\fi} \def\cprime{$'$}
  \def\cprime{$'$} \def\cprime{$'$}
\begin{thebibliography}{10}

\bibitem{AdachiIwataNakatsukasaTakeda:15}
S.~Adachi, S.~Iwata, Y.~Nakatsukasa, and A.~Takeda.
\newblock Solving the trust region subproblem by a generalized eigenvalue
  problem.
\newblock Technical report, Mathematical Engineering, The University of Tokyo,
  2015.

\bibitem{MR2486048}
W.~Ai and S.~Zhang.
\newblock Strong duality for the {CDT} subproblem: a necessary and sufficient
  condition.
\newblock {\em SIAM J. Optim.}, 19(4):1735--1756, 2008.

\bibitem{BoTo:95}
P.T. Boggs and J.W. Tolle.
\newblock Sequential quadratic programming.
\newblock pages 1--51, 1995.

\bibitem{BuAn:11}
S.~Burer and K.M. Anstreicher.
\newblock Second-order-cone constraints for extended trust-region subproblems.
\newblock {\em SIAM J. Optim.}, 23(1):432--451, 2013.

\bibitem{ConGouToi:00}
A.R. Conn, N.I.M. Gould, and Ph.L. Toint.
\newblock {\em Trust-{R}egion {M}ethods}.
\newblock Society for Industrial and Applied Mathematics (SIAM), Philadelphia,
  PA, 2000.

\bibitem{Fia:83}
A.V. Fiacco.
\newblock {\em Introduction to Sensitivity and Stability Analysis in Nonlinear
  Programming}, volume 165 of {\em Mathematics in Science and Engineering}.
\newblock Academic Press, 1983.

\bibitem{FortinWolk:03}
C.~Fortin and H.~Wolkowicz.
\newblock The trust region subproblem and semidefinite programming.
\newblock {\em Optim. Methods Softw.}, 19(1):41--67, 2004.

\bibitem{Gay:81}
D.M. Gay.
\newblock Computing optimal locally constrained steps.
\newblock {\em SIAM J. Sci. Statist. Comput.}, 2:186--197, 1981.

\bibitem{GoRoTho10}
N.I.M. Gould, Daniel~P. Robinson, and H.~Sue Thorne.
\newblock On solving trust-region and other regularised subproblems in
  optimization.
\newblock {\em Math. Program. Comput.}, 2(1):21--57, 2010.

\bibitem{Hager:04}
W.W. Hager and S.~Park.
\newblock Global convergence of {SSM} for minimizing a quadratic over a sphere.
\newblock {\em Math. Comp.}, 74(251):1413--1423, 2005.

\bibitem{HsiaSheu:13}
Y.~Hsia and R.-L. Sheu.
\newblock Trust region subproblem with a fixed number of additional linear
  inequality constraints has polynomial complexity.
\newblock Report, Beihang University, Beijing, China, 2013.

\bibitem{MR3258522}
V.~Jeyakumar and G.~Y. Li.
\newblock Trust-region problems with linear inequality constraints: exact {SDP}
  relaxation, global optimality and robust optimization.
\newblock {\em Math. Program.}, 147(1-2, Ser. A):171--206, 2014.

\bibitem{LRSV:11}
J.~Lampe, M.~Rojas, D.C. Sorensen, and H.~Voss.
\newblock {Accelerating the LSTRS Algorithm}.
\newblock {\em SIAM J. Sci. Comput.}, 33(1):175--194, 2011.

\bibitem{LuPaRo:96}
S.~Lucidi, L.~Palagi, and M.~Roma.
\newblock Quadratic programs with quadratic constraint; characterization of
  {KKT} points and equivalence with an unconstrained problem.
\newblock Technical report, Universita di Roma la Sapienca, 1995.

\bibitem{Mar:94}
J.M. Martinez.
\newblock Local minimizers of quadratic functions on {E}uclidean balls and
  spheres.
\newblock {\em SIAM J. Optim.}, 4(1):159--176, 1994.

\bibitem{MoSo:83}
J.J. Mor\'{e} and D.C. Sorensen.
\newblock Computing a trust region step.
\newblock {\em SIAM J. Sci. Statist. Comput.}, 4:553--572, 1983.

\bibitem{ReWo:94}
F.~Rendl and H.~Wolkowicz.
\newblock A semidefinite framework for trust region subproblems with
  applications to large scale minimization.
\newblock {\em Math. Programming}, 77(2, Ser. B):273--299, 1997.

\bibitem{SalahiFallahi2015}
M.~{Salahi} and S.~{Fallahi}.
\newblock Trust region subproblem with an additional linear inequality
  constraint.
\newblock {\em Optimization Letters}, 2015.

\bibitem{SalahiTaati2015}
M.~{Salahi} and A.~{Taati}.
\newblock A fast eigenvalue approach for solving the trust region subproblem
  with an additional linear inequalityt.
\newblock Technical report, 2015.

\bibitem{S98guide}
J.F. Sturm.
\newblock Using {S}e{D}u{M}i 1.02, a {M}{A}{T}{L}{A}{B} toolbox for
  optimization over symmetric cones.
\newblock {\em Optim. Methods Softw.}, 11/12(1-4):625--653, 1999.
\newblock sedumi.ie.lehigh.edu.

\end{thebibliography}
